\documentclass[3p]{elsarticle}
\usepackage{latexsym,amssymb,amscd}
\usepackage{mathptmx,pifont}
\usepackage{pstricks}
\usepackage{a4wide}
\usepackage{theorem}
\usepackage{hyperref}
\DeclareMathAlphabet{\mathcal}{OMS}{cmsy}{m}{n}

\def\d{\mathsf{d}}

\def\gg{\mathfrak{g}}
\def\RR{\mathbb{R}}
\def\ZZ{\mathbb{Z}}
\def\restrict#1{\,\vrule height1.1ex width.4pt
               depth1.2ex\lower1.0ex\hbox{\scriptsize $\:#1$}}
\newcommand{\operp}{\bigcirc\kern-6.5pt\perp}
\newcommand{\im}{\mathop{\mathrm{im}}\nolimits}
\newcommand{\rank}{\mathop\mathrm{rank}\nolimits}

\makeatletter \@addtoreset{equation}{section}

\makeatother

\theoremheaderfont{\sffamily\bfseries\upshape}
\newtheorem{theorem}{Theorem}[section]
\newtheorem{lemma}[theorem]{Lemma}
\newtheorem{proposition}[theorem]{Proposition}
\newtheorem{corollary}[theorem]{Corollary}

{\theorembodyfont{\normalfont}

\newtheorem{definition}[theorem]{Definition}
\newtheorem{remark}[theorem]{Remark}
\newtheorem{remarks}[theorem]{Remarks}
}

\newenvironment{proof}%
        {\addvspace\baselineskip\noindent {\sc Proof:}\quad}%
        {\hfill \ding{114} \par\addvspace\baselineskip}
        {\addvspace\baselineskip\noindent {\sc Proof of #1:} \quad}%
        {\hfill \ding{144} \par\addvspace\baselineskip}

\begin{document}

\begin{frontmatter}

% Title, authors and addresses

\title{Witten-Hodge theory for manifolds with boundary and equivariant cohomology}

\author{Qusay S.A.~Al-Zamil}
\ead{Qusay.Abdul-Aziz@postgrad.manchester.ac.uk}
\author{James Montaldi}
\ead{j.montaldi@manchester.ac.uk}

\address{School of Mathematics, University of Manchester, Manchester M13 9PL, England}

%\journal{a journal}

\begin{abstract}
We consider a compact, oriented, smooth Riemannian
manifold $M$ (with or without boundary) and we suppose $G$ is a
torus acting by isometries on $M$. Given $X$ in the Lie algebra and
corresponding vector field $X_M$ on $M$, one defines Witten's
inhomogeneous coboundary operator $\d_{X_M} = \d+\iota_{X_M}:
\Omega_G^\pm \to\Omega_G^\mp$ (even/odd invariant forms on $M$) and
its adjoint $\delta_{X_M}$. Witten \cite{Witten} showed that the
resulting cohomology classes have $X_M$-harmonic representatives
(forms in the null space of $\Delta_{X_M} =
(\d_{X_M}+\delta_{X_M})^2$), and the cohomology groups are
isomorphic to the ordinary de Rham cohomology groups of the set
$N(X_M)$ of zeros of $X_M$. Our principal purpose is to extend these
results to manifolds with boundary.  In particular, we define
relative (to the boundary) and absolute versions of the
$X_M$-cohomology and show the classes have representative
$X_M$-harmonic fields with appropriate boundary conditions. To do
this we present the relevant version of the Hodge-Morrey-Friedrichs
decomposition theorem for invariant forms in terms of the operators
$\d_{X_M}$ and $\delta_{X_M}$. We also elucidate the connection
between the $X_M$-cohomology groups and the relative and absolute
equivariant cohomology, following work of Atiyah and Bott. This
connection is then exploited to show that every harmonic field with
appropriate boundary conditions on $N(X_M)$ has a unique
$X_M$-harmonic field on $M$, with corresponding boundary conditions.
Finally, we define the $X_M$-Poincar\'{e} duality angles
between the interior subspaces of $X_M$-harmonic fields on $M$ with
appropriate boundary conditions, following recent work of DeTurck and Gluck.

\begin{keyword}
Hodge theory\sep   manifolds with boundary\sep  equivariant cohomology\sep  Killing vector fields
\MSC[2010]{58J32\sep 57R91\sep 55N91}
\end{keyword}

\end{abstract}

\end{frontmatter}
%%%%%%%%%%%%%%%%%%%%%%%%%%%%%%%%%%%%%%

\section{Introduction}
Throughout we assume $M$ to be a compact oriented smooth Riemannian manifold of dimension $n$, with or without boundary.  For each $k$ we denote by $\Omega^k = \Omega^k(M)$ the space of smooth differential $k$-forms on $M$. The de Rham cohomology of $M$ is defined to be $H^k(M) = \ker\d_k/\im\d_{k-1}$, where $\d_k$ is the restriction of the exterior differential $\d$ to $\Omega^k$.  In other words it is the cohomology of the de Rham complex $(\Omega^*,\d)$. If $M$ has a boundary, then the relative de Rham cohomology $H^k(M,\,\partial M)$   is defined to be the cohomology of the subcomplex $(\Omega^*_D,\d)$ where $\Omega^k_D$ is the space of \emph{Dirichlet} $k$-forms---those satisfying $i^*\omega=0$ where $i:\partial M\ \hookrightarrow M$ is the inclusion of the boundary.

\paragraph{Classical Hodge theory}
Based on the Riemannian structure, there is a natural inner product on each
$\Omega^k$ defined by
\begin{equation}\label{eq:inner product}
\left<\alpha,\,\beta\right> = \int_M\alpha\wedge(\star\beta),
\end{equation}
where $\star:\Omega^k\to\Omega^{n-k}$ is the Hodge star operator \cite{Marsden,Schwarz}. One defines $\delta:\Omega^k\to\Omega^{k-1}$ by
\begin{equation}\label{eq:delta}
\delta\omega = (-1)^{n(k+1)+1} (\star \d\star)\omega.
\end{equation}
If $M$ is boundaryless, this is seen to be the formal adjoint of $\d$ relative to the inner product (\ref{eq:inner product}): $\left<\d\alpha,\,\beta\right> =
\left<\alpha,\,\delta\beta\right>$. The Hodge Laplacian is defined
by $\Delta = (\d+\delta)^2 = \d\delta+\delta \d$, and a form
$\omega$ is said to be \emph{harmonic} if $\Delta\omega=0$.

In the 1930s, Hodge \cite{Hodge} proved the fundamental result that (for $M$ without boundary)
each cohomology class contains a unique harmonic form. A more
precise statement is that, for each $k$,
\begin{equation}\label{eq:Hodge}
\Omega^k(M) = \mathcal{H}^k\oplus \d\Omega^{k-1} \oplus
\delta\Omega^{k+1}.
\end{equation}
The direct sums are orthogonal with respect to the inner product (\ref{eq:inner product}), and the direct sum of the first two subspaces
is equal to the subspace of all closed $k$-forms (that is, $\ker\d_k$). It follows that the Hodge star operator realizes Poincar\'e duality at the level of harmonic forms.

Furthermore, any harmonic form $\omega\in \ker \Delta$ is both closed ($\d\omega=0$) and co-closed ($\delta\omega=0$),  as
\begin{equation}\label{eq:ker=ker}
0=\left<\Delta\omega,\,\omega\right> =
\left<\d\delta\omega,\,\omega\right> +
\left<\delta\d\omega,\,\omega\right> =
\left<\delta\omega,\,\delta\omega\right> +
\left<\d\omega,\,\d\omega\right> = \|\delta\omega\|^2 +
\|\d\omega\|^2.
\end{equation}
For manifolds with boundary this is no longer true, and in general we write
$$\mathcal{H}^k = \mathcal{H}^k(M)=\ker \d \cap\ker\delta.$$
Thus for manifolds without boundary $\mathcal{H}(M)=\ker\Delta$, the space of harmonic forms.

\begin{remark}\label{rmk:harmonic forms are invariant}
An interesting observation which follows from the theorem of Hodge
is the following. If a group $G$ acts on $M$ then there  is an
induced action on each $H^k(M)$, and if this action is trivial (for
example, if $G$ is a connected Lie group) and the action is by
isometries, then each harmonic form is invariant under this action.
\end{remark}

\paragraph{Witten's deformation of  Hodge theory}
Now suppose $K$ is a Killing vector field on $M$ (meaning that
the Lie derivative of the metric vanishes).  Witten \cite{Witten}
defines, for each $s\in\RR$, an operator on differential forms by
$$\d_s := \d + s\, \iota_K\,,$$
where $\iota_K$ is interior multiplication of a form with $K$. This operator is no longer homogeneous in the degree of the form: if $\omega\in\Omega^k(M)$ then $\d_s\omega \in \Omega^{k+1}\oplus\Omega^{k-1}$. Note then that $\d_s:\Omega^{\pm}\to\Omega^{\mp}$, where $\Omega^\pm$ is the space of forms of even ($+$) or odd ($-$) degree. Let us write $\delta_s=\d_s^*$ for the formal adjoint of $\d_s$ (so given by $\delta_s = \delta + s(-1)^{n(k+1)+1} (\star\,\iota_K\star)$ on each homogeneous form of degree $k$).  By Cartan's formula, $\d_s^2 = s\mathcal{L}_K$ (the Lie derivative along $sK$). On the space $\Omega_s^\pm = \Omega^\pm \cap \ker\mathcal{L}_{K}$ of invariant forms, $\d_s^2=0$ so one can define two cohomology groups $H_s^\pm := \ker\d_s^\pm/\im\d_s^\mp$. Witten then defines
$$\Delta_s: = (\d_s+\delta_s)^2:\Omega_{s}^\pm(M)\to\Omega_{s}^\pm(M),$$
(which he denotes $H_s$ as it represents a Hamiltonian operator, but for us this would cause confusion),  and he observes that using standard Hodge theory arguments, there is an isomorphism
\begin{equation}\label{eq:Witten1}
\mathcal{H}_s^\pm := (\ker\Delta_s)^\pm \cong H_s^\pm(M),
\end{equation}
although no details of the proof are given (the interested reader can find details in \cite{my thesis}).  Witten also
shows, among other things, that for $s\neq 0$, the dimensions of
$\mathcal{H}_s^\pm$ are respectively equal to the total even and odd
Betti numbers of the subset $N$ of zeros of $K$, which in
particular implies the finiteness of $\dim\mathcal{H}_s$.
Atiyah and Bott \cite{AB} relate this result of Witten's to
their localization theorem in equivariant cohomology.

\medskip

It is well-known that the group of isometries of a Riemannian manifold (with or without boundary) is compact, so that a Killing vector field generates an action of a torus.  In this light, and because of Remark\,\ref{rmk:harmonic forms are invariant} (and its extension to Witten's setting), Witten's analysis can be cast in the following slightly more general context.

Throughout, we let $G$ be a torus acting by isometries on $M$, with Lie algebra $\gg$, and denote by $\Omega_G=\Omega_G(M)$ the space of smooth $G$-invariant forms on $M$.
Given any $X\in\gg$ we denote the corresponding vector field on $M$
by $X_M$, and following Witten we define $\d_{X_M} = \d +
\iota_{X_M}.$ Then $\d_{X_M}$ defines an operator
$\d_{X_M}:\Omega_G^\pm \to \Omega_G^\mp$, with $\d_{X_M}^2=0$.  For
each $X\in\gg$ there are therefore two corresponding cohomology
groups $H_{X_M}^\pm (M)= \ker\d_{X_M}^\pm/\im\d_{X_M}^\mp$, which we
call $X_M$-cohomology groups, and a corresponding operator we call
the \emph{Witten-Hodge-Laplacian}
$$\Delta_{X_M}=(\d_{X_M}+\delta_{X_M})^2 : \Omega_G^\pm \to \Omega_G^\pm.$$
According to Witten there is an isomorphism $\mathcal{H}_{X_M}^\pm
\cong H_{X_M}^\pm(M)$, where $\mathcal{H}^{\pm}_{X_M}$ is the space
of $X_M$-harmonic forms, that is those forms annihilated by
$\Delta_{X_M}$\,. Of course, Witten's presentation is no less
general than this, and is obtained by putting $X_M=sK$; the only
difference is we are thinking of $X$ as a variable element of $\gg$,
while for Witten varying $s$ only gives a 1-dimensional subspace of
$\gg$ (although one may change $K$ as well).

\medskip

The immediate purpose of this paper is to extend Witten's results to
manifolds with boundary. In order to do this, in Section
\ref{sec:W-H no boundary} we outline the background to Witten's results using classical Hodge theory arguments, which in Section
\ref{sec:W-H with boundary} we extend to deal with the case of
manifolds with boundary. In section \ref{sec:equivariant cohomology},
we describe Atiyah and Bott's localization and its conclusions in
the case of manifolds with boundary, and its relation to
$X_M$-cohomology.  Finally in Section \ref{sec.style of
DeTurck-Gluck}, we extend our results to adapt ideas of  DeTurck and Gluck
\cite{Gluck} and the \emph{Poincar\'{e} duality angles}. Section
\ref{sec:conclusions} provides a few conclusions.

The original motivation for this paper was to adapt to the equivariant setting some recent work of Belishev and Sharafutdinov \cite{Belishev} where they address the classical question,  \emph{ ``To what extent are the topology and geometry of a manifold determined by the Dirichlet-to-Neumann (DN) map''?}
which arises in the scope of inverse problems and reconstructing a
manifold from boundary measurements.  They show that the DN map on the boundary of a Riemannian manifold determines the Betti numbers of the manifold. This paper provides the background necessary for the ``equivariant'' analogue \cite{Q&J} of the results of Belishev and Sharafudtinov.

\medskip

\paragraph{Hodge theory for manifolds with boundary}
In the remainder of this introduction we recall the standard extension of Hodge theory to manifolds with boundary, leading to the Hodge-Morrey-Friedrichs decompositions; details can be found in the book of Schwarz \cite{Schwarz}. The relative de Rham cohomology and the Dirichlet forms are defined at the beginning of the introduction.
One also defines $\Omega^k_N(M) = \left\{\alpha\in\Omega^k(M)\mid i^*(\star\alpha)=0\right\}$ (Neumann boundary condition).
Clearly, the Hodge star provides an isomorphism
$\star:\Omega_D^k\stackrel{\sim}{\longrightarrow}\Omega_N^{n-k}.$
Furthermore, because $\d$ and $i^*$ commute, it follows that $\d$ preserves Dirichlet boundary conditions while $\delta$ preserves Neumann boundary conditions.

As alluded to before, because of boundary terms, the null space of $\Delta$ no longer coincides with the closed and co-closed forms. Elements of $\ker\Delta$ are called \emph{harmonic forms}, while $\omega$ satisfying $\d\omega=\delta\omega=0$ are called \emph{harmonic fields} (following Kodaira); it is clear that every harmonic field is a harmonic form, but the converse is false.
In fact, the space $\mathcal{H}^k(M)$ of harmonic fields is infinite dimensional and so is much too big to represent the cohomology, and to recover the Hodge isomorphism one has to impose boundary conditions.  One restricts $\mathcal{H}^k(M)$ into each of two finite dimensional subspaces, namely $\mathcal{H}_D^k(M)$ and $\mathcal{H}_N^k(M)$ with the obvious meanings (Dirichlet and Neumann harmonic $k$-fields, respectively).  There are therefore two different candidates for harmonic representatives when the boundary is present.

The Hodge-Morrey decomposition \cite{Morrey} states that
$$ \Omega^k(M) = \mathcal{H}^k(M) \oplus \d\Omega_D^{k-1} \oplus \delta\Omega_N^{k+1}.$$
(We make a more precise functional analytic statement below.)
This decomposition is again orthogonal with respect to the inner
product given above. Friedrichs \cite{Friedrichs} subsequently
showed that
$$\mathcal{H}^k = \mathcal{H}^k_D\oplus \mathcal{H}^k_{\mathrm{co}};\qquad \mathcal{H}^k = \mathcal{H}^k_N\oplus \mathcal{H}^k_{\mathrm{ex}}$$
where $\mathcal{H}^k_{\mathrm{ex}}$ are the exact harmonic fields
and $\mathcal{H}^k_{\mathrm{co}}$ the coexact ones (that is,  $\mathcal{H}^k_{\mathrm{ex}} = \mathcal{H}^k\cap\d\Omega^{k-1}$ and $\mathcal{H}^k_{\mathrm{co}} = \mathcal{H}^k\cap\delta\Omega^{k+1}$). These give the
orthogonal \emph{Hodge-Morrey-Friedrichs} \cite{Schwarz}
decompositions,
\begin{eqnarray*}
\Omega^k(M) &=& \d\Omega_D^{k-1} \oplus \delta\Omega_N^{k+1} \oplus \mathcal{H}^k_D\oplus \mathcal{H}^k_{\mathrm{co}}\\
&=& \d\Omega_D^{k-1} \oplus \delta\Omega_N^{k+1} \oplus
\mathcal{H}^k_N\oplus \mathcal{H}^k_{\mathrm{ex}}.
\end{eqnarray*}
The two decompositions are related by the Hodge star operator. The
consequence for cohomology is that each class in $H^k(M)$ is
represented by a unique harmonic field in $\mathcal{H}^k_N(M)$, and
each relative class in $H^k(M,\partial M)$ is represented by a
unique harmonic field in  $\mathcal{H}^k_D(M)$. Again, the Hodge
star operator acts as Poincar\'e duality (or rather Poincar\'e-Lefschetz duality) on the harmonic fields, sending Dirichlet fields to Neumann fields. And as in remark~\ref{rmk:harmonic forms are invariant}, if a group acts by isometries on $(M,\partial M)$ in a manner that is trivial on the cohomology, then the harmonic fields are invariant.

\medskip

In this paper, we suppose $G$ is a compact connected Abelian Lie group (a torus) acting by isometries on $M$, with Lie algebra $\mathfrak{g}$, and we let $X\in\mathfrak{g}$. If $M$ has a boundary then the $G$-action necessarily restricts to an action on the boundary and $X_M$ must therefore be tangent to the boundary.  We denote by $\Omega_G=\Omega_G(M)$ the set of invariant forms on $M$: $\omega\in\Omega_G$ if $g^*\omega=\omega$ for all $g\in G$; in particular if $\omega$ is invariant then the Lie derivative $\mathcal{L}_{X_M}\omega=0$.
Note that because the action preserves the metric and the
orientation it follows that, for each $g\in G$, $\star(g^*\omega) =
g^*(\star\omega)$, so if $\omega\in\Omega_G$ then
$\star\omega\in\Omega_G$.

\medskip
Remark on typesetting: Since the letter H plays three roles in this paper, we use three different typefaces: a script $\mathcal{H}$ for harmonic fields, a sans-serif $\mathsf{H}$ for Sobolev spaces and a normal (italic) $H$ for cohomology. We hope that will prevent any confusion.

%%%%%%%%%%%%%%%%%%%%%%%%%%%%%%%%%%%%%%%%%%%%%%%%%%
\section{Witten-Hodge theory for manifolds without boundary}
\label{sec:W-H no boundary}

In this section we summarize the functional analysis behind Witten's results \cite{Witten},  details can be found in the first author's thesis \cite{my thesis}.
These are  needed in the next section for manifolds with boundary.  We continue to use the notation from the introduction, notably the manifold $M$ (which in this section has no boundary) and the torus $G$.

Fix an element $X\in\gg$. The associated vector field on $M$ is $X_M$, and using this one defines Witten's inhomogeneous operator $\d_{X_M}: \Omega_G^\pm \to
\Omega_G^\mp,\; \d_{X_M}\omega=\d\omega+\iota_{X_M}\omega$, and the corresponding operator (cf.\ eq.~(\ref{eq:delta}))
$$\delta_{X_M} = (-1)^{n(k+1)+1}\star\d_{X_M}\star = \delta + (-1)^{n(k+1)+1}\star\iota_{X_M}\star$$
(which is the operator adjoint to $\d_{X_M}$ by eq. (\ref{eq.2.16})
below). The resulting \emph{Witten-Hodge-Laplacian} is
$\Delta_{X_M}:\Omega_G^\pm \to \Omega_G^\pm$ defined by
$\Delta_{X_M}=(\d_{X_M}+\delta_{X_M})^2 = \d_{X_M}\delta_{X_M} +
\delta_{X_M}\d_{X_M}$. We write the space of $X_M$-harmonic fields
$$\mathcal{H}_{X_M} = \ker\d_{X_M}\cap\ker\delta_{X_M}\,,$$
which for manifolds without boundary satisfies  $\mathcal{H}_{X_M} =\ker\Delta_{X_M}$. The last equality follows for the same reason as for ordinary Hodge theory, namely the argument in (\ref{eq:ker=ker}), with $\Delta$
replaced by $\Delta_{X_M}$ etc.

As is conventional, define  $\int_M\omega=0$ if  $\omega\in\Omega^k(M)$ with $k\neq
n$. So, for any form $\omega\in\Omega(M)$ one has $\int_M
\iota_{X_M}\omega=0$ as $\iota_{X_M}\omega$ has no term of degree
$n$, and the following equation (\ref{eq.Stokes' theorem}) follows
from the
ordinary Stokes' theorem. For future use, we allow $M$ to have a
boundary.
\begin{equation}\label{eq.Stokes' theorem}
    \int_M \d_{X_{M}}\omega= \int_{\partial M} i^{*}\omega.
\end{equation}

For each space $\Omega$ of smooth differential forms on $M$, and
each $s\in\RR$, we write $\mathsf{H^{s}}\Omega$ for the completion
of $\Omega$ under an appropriate Sobolev norm. It is not hard to
prove a Green's formula in terms of $\d_{X_{M}}$ and $\delta_{X_{M}}$
which states that for $\alpha ,\beta \in
\mathsf{H^{1}}\Omega_{G}$,
\begin{equation}\label{eq.2.16}
\langle \d_{X_{M}}\alpha,\beta\rangle=\langle
\alpha,\delta_{X_{M}}\beta\rangle+\int_{\partial {M}} i^{*} (\alpha
\wedge \star\beta)\,,
\end{equation}

Returning now to the case of a manifold without boundary, we obtain the following.

\begin{theorem}\label{thm:self-adjoint elliptic no boundary}
\begin{enumerate}
\item The Witten-Hodge-Laplacian $\Delta_{X_M}$ is a self-adjoint elliptic
operator.\\
\item The following is an orthogonal decomposition
$$\Omega_{G}^\pm = \mathcal{H}_{X_M}^\pm \oplus \d_{X_M}\Omega_G^\mp \oplus \delta_{X_M}\Omega_G^\mp.$$
The orthogonality is with respect to the  $L^2$ inner product.
\end{enumerate}
\end{theorem}

Part (2) is the analogue of the Hodge decomposition theorem, and is a standard consequence of the fact that $\Delta_{X_M}$ is self-adjoint.  The first two summands give the $X_M$-closed forms.

Every elliptic operator on a compact manifold is a Fredholm
operator, so has finite dimensional kernel and cokernel, and closed
range. Therefore the set of $X_M$-harmonic (even/odd) forms
$\mathcal{H}_{X_M}^\pm=(\ker \Delta_{X_M})^\pm$ is finite
dimensional. One concludes with the analogue of Hodge's theorem

\begin{proposition}\label{unique $X_M$-harmonic representative}  $H^{\pm}_{X_M}(M)\cong \mathcal{H}_{X_M}^\pm$, and in particular every $X_M$-cohomology class has a unique $X_M$-harmonic representative.
\end{proposition}

The Hodge star operator gives a form of Poincar\'{e} duality in terms of
$X_M$-cohomology:
$$ H_{X_M}^{n-\pm}(M)\cong H_{X_M}^\pm(M).$$
Since Hodge star takes harmonic forms to harmonic forms, this Poincar\'e duality is realized at the level of harmonic forms. The full details are given in \cite{my thesis}. Here and elsewhere we write $n-\pm$ for the parity
(modulo 2) resulting from subtracting an even/odd number from $n$.

Let $N(X_M)$ be the set of zeros of $X_M$, and $j:N(X_M)\hookrightarrow M$ the inclusion. As observed by Witten, on $N(X_M)$ one has $X_M=0$, so that $j^*\d_{X_M}\omega = \d(j^*\omega)$, and in particular if $\omega$ is $X_M$-closed then
its pullback to $N(X_M)$ is closed in the usual (de Rham) sense. And
exact forms pull back to exact forms. Consequently, pullback defines
a natural map $H_{X_M}^\pm(M)\to H^\pm(N(X_M))$, where
$H^\pm(N(X_M))$ is the direct sum of the even/odd de Rham cohomology groups
of $N(X_M)$.

\begin{theorem}[Witten \cite{Witten}] \label{thm:fixed point set} The pullback to $N(X_M)$ induces an isomorphism between the $X_M$-cohomology groups $H_{X_M}^\pm(M)$ and the cohomology groups $H^\pm(N(X_M))$.
\end{theorem}

Witten gives a fairly explicit proof of this theorem by extending
closed forms on $N(X_M)$ to $X_M$-closed forms on $M$. Atiyah and
Bott \cite{AB} give a proof using their localization theorem in
equivariant cohomology  which we discuss, and adapt to the case of
manifolds with boundary, in Section \ref{sec:equivariant
cohomology}.

\begin{remark}
Extending remark \,\ref{rmk:harmonic forms are invariant}, suppose $X$ generates the torus $G(X)$, and $G$ is a larger torus containing  $G(X)$ and acting on $M$ by isometries. Then the action
of $G$ preserves $X_M$.  It follows that $G$ acts trivially on the
de Rham cohomology of $N(X_M)$, and hence on the $X_M$-cohomology of
$M$, and consequently on the space of $X_M$-harmonic forms. In other
words, $\mathcal{H}_{X_M}^\pm\subset\Omega_{G}^\pm$. There is
therefore no loss in considering just forms invariant under the
action of the larger torus in that the $X_M$-cohomology, or the
space of $X_M$-harmonic forms, is independent of the choice of
torus, provided it contains $G(X)$.
\end{remark}

%%%%%%%%%%%%%%%%%%%%%%%%%%%%%%%%%%%%%%%%%%%%%%%%%%%
\section{Witten-Hodge theory for manifolds with boundary}
\label{sec:W-H with boundary}

In this section we extend the results and methods of Hodge theory
for manifolds with boundary to study the $X_M$-cohomology and the
space of $X_M$-harmonic forms and fields for manifolds with
boundary. As for ordinary (singular) cohomology, there are both
absolute and relative $X_M$-cohomology groups.
From now on our manifold will be with boundary and as before
$i:\partial M \hookrightarrow M$ denotes the inclusion of the
boundary, and $G$ is a torus acting by isometries on $M$.

\subsection{The difficulties if the boundary is present}
Firstly, $\d_{X_{M}}$ and $\delta_{X_{M}}$ are no longer adjoint
because the boundary terms arise when we integrate by parts, and then
$\Delta_{X_{M}}$ will not be self-adjoint. In addition, the space of
all harmonic fields is infinite dimensional and there is no reason
to expect the $X_{M}$-harmonic fields $\mathcal{H}_{X_{M}}(M)$ to be
any different. To overcome these problems, at the beginning we
follow the method which is used to solve this problem in the
classical case, i.e.\ with $\d$ and $\delta$, by
imposing certain boundary conditions on our invariant forms
$\Omega_{G}(M)$, as described in  \cite{Schwarz}. Hence we make the following definitions.

\begin{definition}
(1) We define the following two sets of smooth invariant forms on the manifold $M$ with boundary and with action of the torus $G$
\begin{eqnarray}
  \Omega_{G,D} &=& \Omega_{G}\cap \Omega_{D} \ = \{\omega\in\Omega_G\mid i^*\omega=0\} \label{eq.2.12}\\
   \Omega_{G,N} &=& \Omega_{G}\cap \Omega_{N}\ = \{\omega\in\Omega_G\mid i^*(\star\omega)=0\}\label{eq.2.13}
\end{eqnarray}
and the spaces $\mathsf{H}^{s}\Omega_{G,D}$ and $\mathsf{H}^{s}\Omega_{G,N}$ are the corresponding closures with respect to suitable Sobolev norms, for $s>\frac12$. This can be refined to take into account the parity of the forms, so defining $\Omega_{G,D}^\pm$ etc. Since $\omega\in\Omega^k$ implies $\star\omega \in \Omega^{n-k}$ we write that for $\omega\in\Omega_{G}^\pm$ we have $\star\omega \in\Omega_G^{n-\pm}$.

\noindent(2) We define two subspaces of $X_M$-harmonic fields,
\begin{eqnarray}
  \mathcal{H}_{X_{M},D}(M) &=& \{\omega\in \mathsf{H^{1}}\Omega_{G,D}\mid \d_{X_{M}}\omega=0, \,\delta_{X_{M}}\omega=0 \}\label{eq.2.14}\\
  \mathcal{H}_{X_{M},N}(M) &=& \{\omega\in \mathsf{H^{1}}\Omega_{G,N}\mid \d_{X_{M}}\omega=0,\, \delta_{X_{M}}\omega=0 \}\label{eq.2.15}
\end{eqnarray}
which we call Dirichlet and Neumann $X_{M}$-harmonic fields, respectively.  We will show below that these forms are smooth. Clearly, the Hodge star operator $\star$ defines an isomorphism  $\mathcal{H}_{X_{M},D}(M) \cong \mathcal{H}_{X_{M},N}(M)$. Again, these can be refined to take the parity into account, defining $  \mathcal{H}_{X_{M},D}^\pm(M)$ etc.
\end{definition}

As for ordinary Hodge theory, on a manifold with boundary one has to distinguish between $X_{M}$-harmonic \emph{forms} (i.e.\ $\ker\Delta_{X_{M}}$) and $X_{M}$-harmonic \emph{fields} (i.e.\ $\mathcal{H}_{X_{M}}(M)$) because they are not equal: one has $\mathcal{H}_{X_{M}}(M)\subseteq \ker\Delta_{X_{M}}$ but not conversely. The following proposition shows the conditions on $\omega$ to be fulfilled in order to ensure $\omega\in\ker\Delta_{X_{M}} \Longrightarrow \omega\in \mathcal{H}_{X_{M}}(M)$ when $\partial M\neq \emptyset$.

\begin{proposition}\label{pro.2.6}
If $\omega \in \Omega_{G}(M)$ is an $X_{M}$-harmonic form
(i.e.\ $\Delta_{X_{M}}\omega=0$) and in addition any one of the following four pairs of boundary conditions is satisfied then $\omega\in\mathcal{H}_{X_{M}}(M)$.
$$
  \begin{array}{llll}
    (1)& i^{*} \omega = 0 ,\; i^{*} (\star \omega)=0; &
    (2)& i^{*} \omega = 0 ,\; i^{*}(\delta_{X_{M}}\omega)=0; \\[4pt]
    (3)& i^{*} (\star \omega) = 0 ,\; i^{*}(\star \d_{X_{M}}\omega)=0; \quad  &
    (4)& i^{*} (\delta_{X_{M}}\omega) = 0 ,\; i^{*}(\star \d_{X_{M}}\omega)=0.
  \end{array}
$$

\end{proposition}
\begin{proof}
Because $\Delta_{X_{M}}\omega=0$, one has $\langle \Delta_{X_{M}}
\omega,\omega\rangle=0$.  Now applying Green's formula~(\ref{eq.2.16}) to this
and using any of these conditions (1)--(4) ensures $\omega$ is an
$X_{M}$-harmonic field.
\end{proof}

\begin{remark}\label{r.6'}
An averaging argument shows that $\mathsf{H^{1}}\Omega_{G},_{D}$ and
$\mathsf{H^{1}}\Omega_{G},_{N}$ are dense in $L^{2}\Omega_{G}$,
because the corresponding statements hold for the spaces of all (not only invariant) forms.
\end{remark}

\subsection{Elliptic boundary value problem}\label{subsec. Elliptic BVP}

The essential ingredients that Schwarz \cite{Schwarz} needs to prove the classical \emph{Hodge-Morry-Friedrichs} decomposition are Gaffney's inequality and his Theorem 2.1.5. However, these results do not appear to extend to the context
of $\d_{X_M}$ and $\delta_{X_M}$. Therefore, we use a different approach to
overcome this problem, based on the ellipticity of a certain
boundary value problem (\textsc{bvp}), namely (\ref{eq.2.18}) below. This theorem represents the keystone to extending the Hodge-Morrey and Friedrichs decomposition theorems to the present setting and thence to extending Witten's results to manifolds with boundary.

Consider the \textsc{bvp}
\begin{equation} \label{eq.2.18}
\left\{\begin{array}{rcl}
  \Delta_{X_{M}}\omega &=& \eta \quad \textrm{on} \quad M \\
   i^{*}\omega &=& 0 \quad \textrm{on} \quad \partial M\\
    i^{*}(\delta_{X_{M}}\omega)&=& 0 \quad  \textrm{on} \quad \partial M.
\end{array}\right.
\end{equation}

\begin{theorem}\label{thm.2.5}\quad
\begin{enumerate}
\item The \textsc{bvp} (\ref{eq.2.18}) is elliptic in the sense of Lopatinski\v{\i}-\v{S}apiro,  where $\Delta_{X_{M}}:\Omega_{G}(M)\longrightarrow \Omega_{G}(M)$.
\item The \textsc{bvp} (\ref{eq.2.18}) is Fredholm of index 0.
\item All $\omega\in\mathcal{H}_{X_M,D}\cup\mathcal{H}_{X_M,N}$ are smooth.
\end{enumerate}
\end{theorem}

\begin{proof}
(1)
We can see that $\Delta$ and $\Delta_{X_M}$ have the same principal
symbol as $\Delta_{X_M}-\Delta$ is a first order differential
operator; indeed, 
$$\Delta_{X_M} = \Delta+(-1)^{n(k+1)+1}(\d\star
\iota_{X_{M}}\star+\star \iota_{X_{M}}\star \d + \star
\iota_{X_{M}}\star \iota_{X_{M}}+ \iota_{X_{M}}\star
\iota_{X_{M}}\star) +\iota_{X_{M}} \delta+ \delta \iota_{X_{M}}\,.$$
Similarly, expanding the second boundary condition gives
$$\delta_{X_M} = \delta + (-1)^{n(k+1)+1}\star\iota_{X_M}\star$$
so $\delta_{X_M}$ and $\delta$ have the same first-order part. Hence our \textsc{bvp}~(\ref{eq.2.18})  has the same principal symbol as
the  \textsc{bvp}
\begin{equation} \label{eq.2.20}
\left\{\begin{array}{rcl}
  \Delta\epsilon &=& \xi \quad \textrm{on} \quad M \\
   i^{*}\epsilon &=& 0 \quad \textrm{on} \quad \partial M\\
    i^{*}(\delta\epsilon)&=& 0 \quad  \textrm{on} \quad \partial M
\end{array}\right.
\end{equation}
for $\epsilon,\,\xi \in \Omega(M)$, because the principal symbol does not change when terms of lower order are added to the operator. However the \textsc{bvp}~(\ref{eq.2.20}) is elliptic in the sense of Lopatinski\v{\i}-\v{S}apiro conditions \cite{Lars,Schwarz}, and thus  so is (\ref{eq.2.18}).

\medskip

\noindent(2) From part (1), since the \textsc{bvp}~(\ref{eq.2.18})
is elliptic, it follows that the \textsc{bvp}~(\ref{eq.2.18}) is a Fredholm operator and the
regularity theorem holds, see for example Theorem 1.6.2 in \cite{Schwarz} or Theorem 20.1.2 in \cite{Lars}. In addition, we observe that the only
differences between \textsc{bvp}~(\ref{eq.2.20}) and our
\textsc{bvp}~(\ref{eq.2.18}) are all lower order operators and it is
proved in \cite{Schwarz} that the index of
\textsc{bvp}~(\ref{eq.2.20}) is zero but Theorem 20.1.8 in
\cite{Lars} asserts generally that if the difference between two
\textsc{bvp}s are just lower order operators then they must have
the same index. Hence, the index of the \textsc{bvp}~(\ref{eq.2.18})
must be zero.

\medskip

\noindent(3)
Let $\omega\in\mathcal{H}_{X_M,D}\cup\mathcal{H}_{X_M,N}$.  If $\omega
\in \mathcal{H}_{X_M,D}$ then it satisfies the \textsc{bvp}~(\ref{eq.2.18}) with $\eta=0$, so by the regularity properties of elliptic \textsc{bvp}s, the
smoothness of $\omega$ follows. If on the other hand $\omega \in \mathcal{H}_{X_M,N}$ then $\star\omega\in \mathcal{H}_{X_{M},D}$ which is therefore smooth and consequently $\omega=\pm\star(\star\omega)$ is smooth as well.
\end{proof}

We consider the resulting operator obtained by restricting $\Delta_{X_M}$ to the subspace of smooth invariant forms satisfying the boundary conditions
\begin{equation}\label{eq.2.21}
\overline{\Omega}_{G}(M)=\{\omega \in \Omega_{G}(M)
 \mid i^{*}\omega = 0, \,i^{*}(\delta_{X_{M}}\omega)= 0\}
\end{equation}

Since the trace map $i^*$ is well-defined on $\mathsf{H}^s\Omega_G$
for $s>1/2$ it follows that it makes sense to consider
$\mathsf{H^{2}}\overline{\Omega}_{G}(M)$, which is a closed subspace
of $\mathsf{H^{2}}\Omega_{G}(M)$ and hence a Hilbert space. For
simplicity, we rewrite our \textsc{bvp}~(\ref{eq.2.18}) as
follows:  consider the restriction/extension of
$\Delta_{X_M}$ to this space:
$$A=\Delta_{X_{M}}\restrict{\mathsf{H^{2}}\overline{\Omega}_{G}(M)}:
\mathsf{H^{2}}\overline{\Omega}_{G}(M)\longrightarrow
L^{2}\Omega_{G}(M).
$$
and consider the \textsc{bvp},
\begin{equation}\label{eq.2.22}
A\omega=\eta
\end{equation}
for $\omega\in\mathsf{H^{2}}\overline{\Omega}_{G}(M)$ and $\eta\in
L^{2}\Omega_{G}(M)$ instead of \textsc{bvp} (\ref{eq.2.18}) which are
in fact compatible. In addition, from Theoremsa~\ref{thm.2.5} we
deduce that $A$ is an elliptic and Fredholm operator and
\begin{equation}\label{index}
    \mathrm{index}(A)=\dim(\ker A)-\dim(\ker A^*)=0
\end{equation}
where $A^*$ is the adjoint operator of $A$.

From Green's formula (eq. (\ref{eq.2.16})) we deduce the following
property.

\begin{lemma}\label{L.2.33}
$A$ is $L^2$-self-adjoint on $\mathsf{H^{2}}\overline{\Omega}_{G}(M)$, meaning that for all $\alpha,\beta\in\mathsf{H^{2}}\overline{\Omega}_{G}(M)$ we have
$$\left<A\alpha,\,\beta\right> = \left<\alpha,\,A\beta\right>,$$
where $\left<-,-\right>$ is the $L^2$-pairing.
\end{lemma}

\begin{theorem}\label{tm.2.6'}
The space $ \mathcal{H}_{X_{M},D}(M)$ is finite dimensional and
\begin{equation}\label{eq.2.24}
     L^{2}\Omega_{G}(M)= \mathcal{H}_{X_{M},D}(M) \oplus \mathcal{H}_{X_{M},D}(M)^{\perp}.
  \end{equation}
\end{theorem}
\begin{proof}
We begin by showing that $\ker A = \mathcal{H}_{X_{M},D}(M)$. It is
clear that $\mathcal{H}_{X_{M},D}(M)\subseteq \ker A$, so we need
only prove that $\ker A \subseteq \mathcal{H}_{X_{M},D}(M)$.

Let $\omega \in \ker A$.  Then $\omega$ satisfies the \textsc{bvp} (\ref{eq.2.18}). Therefore, by condition (2) of Proposition~\ref{pro.2.6}, it follows that $\omega \in  \mathcal{H}_{X_{M},D}(M)$, as required.

Now, $\ker A = \mathcal{H}_{X_{M},D}(M)$ but $\dim\ker A$ is finite, so that $\dim\mathcal{H}_{X_{M},D}(M)<\infty$. This implies that $\mathcal{H}_{X_{M},D}(M)$ is a closed subspace of the Hilbert space $L^{2}\Omega_{G}(M)$, hence eq.~(\ref{eq.2.24}) holds.
\end{proof}

\begin{theorem}\label{tm.2.6}
\begin{equation}\label{eq.2.23}
\mathrm{Range}(A) = \mathcal{H}_{X_{M},D}(M)^{\perp}
\end{equation}
where $\perp$ denotes the orthogonal complement in $L^2\Omega_{G}(M)$.
\end{theorem}

\begin{proof}
 Firstly, we should observe that eq.~(\ref{index}) asserts that $\ker A\cong \ker A^*$ but
Theorem \ref{tm.2.6'} shows that $\ker A = \mathcal{H}_{X_{M},D}(M)$, thus
\begin{equation}\label{eq.coker}
    \ker A^*\cong  \mathcal{H}_{X_{M},D}(M)
\end{equation}

Since $\mathrm{Range}(A)$ is closed in $L^2\Omega_{G}(M)$ because
$A$ is Fredholm operator, it follows from the closed range theorem
in Hilbert spaces that
\begin{equation}\label{eq.closed range }
\mathrm{Range}(A) = (\ker A^*)^{\perp} \quad \equiv \quad
\mathrm{Range}(A)^{\perp} = \ker A^*
\end{equation}
Hence, we just need to prove that $\ker A^*=
\mathcal{H}_{X_{M},D}(M)$, and to show that we need first to prove
\begin{equation}\label{eq.subset1}
\mathrm{Range}(A)\subseteq \mathcal{H}_{X_{M},D}(M)^{\perp}.
\end{equation}
So, if $\alpha\in\mathsf{H^{2}}\overline{\Omega}_{G}(M)$ and $\beta
\in\mathcal{H}_{X_{M},D}(M) $ then applying Lemma \ref{L.2.33} gives
$$\left<A\alpha,\,\beta\right> =0$$
hence, eq.~(\ref{eq.subset1}) holds. Moreover, equations~(\ref{eq.closed
range }) and (\ref{eq.subset1}) and the closedness of
$\mathcal{H}_{X_{M},D}(M)$ imply
\begin{equation}\label{eq.subset2}
  \mathcal{H}_{X_{M},D}(M) \subseteq \ker A^*
\end{equation}
but eq.~(\ref{eq.coker}) and eq.~(\ref{eq.subset2}) force $\ker A^*=
\mathcal{H}_{X_{M},D}(M)$. Hence, $\mathrm{Range}(A) =
\mathcal{H}_{X_{M},D}(M)^{\perp}$.
\end{proof}

Following \cite{Schwarz}, we denote the $L^{2}$-orthogonal
complement of $\mathcal{H}_{X_{M},D}(M)$ in the space
$\mathsf{H^{2}}\Omega_{G,D}$ by
\begin{equation}\label{eq.2.27}
\mathcal{H}_{X_{M},D}(M)^{\operp} = \mathsf{H^{2}}\Omega_{G,D}\cap
\mathcal{H}_{X_{M},D}(M)^{\perp}
\end{equation}
(although in \cite{Schwarz} it denotes $\mathsf{H}^1$-forms rather than $\mathsf{H}^2$).

\begin{proposition}\label{pro.2.7}
For each $\eta \in \mathcal{H}_{X_{M},D}(M)^{\perp}$   there is a
unique differential form $\omega \in \mathcal{H}_{X_{M},D}(M)^{\operp}$ satisfying the \textsc{bvp}~(\ref{eq.2.18}).
\end{proposition}

\begin{proof}
Let $\eta \in \mathcal{H}_{X_{M},D}(M)^{\perp}$.  Because of Theorem
(\ref{tm.2.6}) there is a differential form $\gamma \in
\mathsf{H^{2}}\overline{\Omega}_{G}(M)$ such that $\gamma$ satisfies
the \textsc{bvp}~(\ref{eq.2.18}). Since $ \gamma \in
\mathsf{H^{2}}\overline{\Omega}_{G}(M)\subseteq L^{2}\Omega_{G}(M)$
then there are unique differential forms $\alpha \in
\mathcal{H}_{X_{M},D}(M)$ and $\omega \in
\mathcal{H}_{X_{M},D}(M)^{\perp} $ such that $\gamma=\alpha+\omega$
because of eq.~(\ref{eq.2.24}).

Since $\gamma$ satisfies the \textsc{bvp}~(\ref{eq.2.18}) it follows
that $\omega$ satisfies the \textsc{bvp}~(\ref{eq.2.18}) as well
because $\alpha \in \mathcal{H}_{X_{M},D}(M) =
\ker(\Delta_{X_{M}}\restrict{\mathsf{H^{2}}\overline{\Omega}_{G}(M)})$.
Since $\omega=\gamma-\alpha$, it follows that $\omega \in
\mathsf{H^{2}}\Omega_{G,D}$ , hence $\omega \in
\mathcal{H}_{X_{M},D}(M)^{\operp}$ and it is unique.
\end{proof}

\begin{remarks}\label{r.8}
\begin{enumerate}
\item[(1)]  $\omega$  satisfying the \textsc{bvp}~(\ref{eq.2.18}) in
Proposition~\ref{pro.2.7} can be recast to the condition
\begin{equation}\label{eq.2.28}
  \langle \d_{X_{M}} \omega,\, \d_{X_{M}}\xi\rangle + \langle\delta_{X_{M}}\omega,\, \delta_{X_{M}}\xi\rangle
  =\langle \eta,\xi\rangle,\quad
\forall  \xi \in \mathsf{H^{1}}\Omega_{G,D}
\end{equation}

\item[(2)]All the results above can be recovered but in terms of
$\mathcal{H}_{X_{M},N}(M)$ because the Hodge star operator defines
an isomorphism $L^2\Omega_G\cong L^2\Omega_G$ which restricts to
$\mathcal{H}_{X_{M},D}(M)\cong \mathcal{H}_{X_{M},N}(M)$.
\end{enumerate}
\end{remarks}

\subsection{Decomposition theorems}

The results above provide the basic ingredients needed to
extend the Hodge-Morrey and Freidrichs decompositions arising for
Hodge theory on manifolds with boundary, to the present setting with
$\d_{X_M}$ and $\delta_{X_M}$.  Depending on these results, the
proofs in this subsection rely heavily on the analogues of the
corresponding statements for the usual Laplacian $\Delta$ on a
manifold with boundary, as described in the book of Schwarz
\cite{Schwarz}. Therefore, we omit the proofs here while full
details are given in the first author's thesis \cite{my thesis}.

\begin{definition}\label{d.2.4}
Define the following two sets of invariant exact and coexact forms on $M$,
\begin{eqnarray*}
    \mathcal{E}_{X_M}(M)&=&\{\d_{X_{M}} \alpha \mid  \alpha \in
    \mathsf{H^{1}}\Omega_{G,D}\}\subseteq L^{2}\Omega_{G}(M),\\
     \mathcal{C}_{X_M}(M) &=& \{\delta_{X_{M}} \beta \mid  \beta \in
    \mathsf{H^{1}}\Omega_{G,N}\}\subseteq L^{2}\Omega_{G}(M).
\end{eqnarray*}
Clearly,  $\mathcal{E}_{X_M}(M)\perp \mathcal{C}_{X_M}(M)$ because
of eq.~(\ref{eq.2.16}).  We denote by $L^{2}\mathcal{H}_{X_{M}}(M) =
\overline{\mathcal{H}_{X_{M}}(M)}$ the $L^{2}$-closure of the space
$\mathcal{H}_{X_{M}}(M)$.
\end{definition}

\begin{proposition}[Algebraic decomposition and $L^{2}$-closedness]\label{pro.2.8}
\begin{enumerate}
\item [(a)]Each $\omega \in L^{2}\Omega_{G}(M)$ can be split uniquely into
$$\omega=\d_{X_{M}} \alpha_{\omega}+\delta_{X_{M}} \beta_\omega
+\kappa_{\omega}$$ where $\d_{X_{M}} \alpha_{\omega} \in
\mathcal{E}_{X_M}(M)$ , $\delta_{X_{M}} \beta_{\omega} \in
\mathcal{C}_{X_M}(M)$ and $\kappa_{\omega} \in
(\mathcal{E}_{X_M}(M)\oplus \mathcal{C}_{X_M}(M))^{\perp}$.

\item[(b)]The spaces $\mathcal{E}_{X_M}(M)$ and $\mathcal{C}_{X_M}(M)$ are closed subspaces of $L^{2}\Omega_{G}(M)$.
\item[(c)] Consequently there is the following orthogonal decomposition
$$
L^{2}\Omega_{G}(M)=\mathcal{E}_{X_M}(M)\oplus
\mathcal{C}_{X_M}(M)\oplus(\mathcal{E}_{X_M}(M)\oplus
\mathcal{C}_{X_M}(M))^{\perp}
$$
\end{enumerate}
\end{proposition}

Now we can present the main theorems for this section; all orthogonality is with respect to the $L^2$ inner product.

\begin{theorem}[$X_M$-Hodge-Morrey decomposition theorem] \quad \label{thm:Hodge-Morrey}
The following is an orthogonal direct sum decomposition:
$$L^{2}\Omega_{G}(M)=\mathcal{E}_{X_M}(M)\oplus
\mathcal{C}_{X_M}(M)\oplus L^{2}\mathcal{H}_{X_{M}}(M)
$$
\end{theorem}

\begin{theorem}[$X_M$-Friedrichs decomposition theorem]\label{thm:X_M-Friedrichs}
The space $\mathcal{H}_{X_{M}}(M)\subseteq \mathsf{H^{1}}\Omega_{G}(M)$ of\/
$X_{M}$- harmonic fields can respectively be decomposed as orthogonal direct sums into
\begin{eqnarray*}
 \mathcal{H}_{X_{M}}(M)&=& \mathcal{H}_{X_{M},D}(M)\oplus \mathcal{H}_{X_{M},\mathrm{co}}(M)\\ %\label{eq.2.35}\\
  \mathcal{H}_{X_{M}}(M)&=& \mathcal{H}_{X_{M},N}(M)\oplus\mathcal{H}_{X_{M},\mathrm{ex}}(M), %\label{eq.2.36}
\end{eqnarray*}
where the right hand terms are the $X_M$-coexact and exact harmonic forms respectively:
\begin{eqnarray*}
  \mathcal{H}_{X_{M},\mathrm{co}}(M) &=& \{ \eta \in \mathcal{H}_{X_{M}}(M)\mid \eta=\delta_{X_{M}}\alpha\}\\ %\label{eq.2.37} \\
  \mathcal{H}_{X_{M},\mathrm{ex}}(M)&=& \{ \xi \in \mathcal{H}_{X_{M}}(M)\mid \xi=\d_{X_{M}}\sigma\}%\label{eq.2.38}
\end{eqnarray*}
For $L^{2}\mathcal{H}_{X_{M}}(M)$ these decompositions are valid
accordingly.
\end{theorem}

Combining Theorems~\ref{thm:Hodge-Morrey} and \ref{thm:X_M-Friedrichs} gives the following.

\begin{corollary}[The $X_M$-Hodge-Morrey-Friedrichs decompositions]\label{co.2.6}
The space $ L^{2}\Omega_{G}(M)$ can be decomposed into $L^{2}$-orthogonal direct sums as follows:
\begin{eqnarray*}
  L^{2}\Omega_{G}(M) &=&\mathcal{E}_{X_M}(M)\oplus \mathcal{C}_{X_M}(M)\oplus
\mathcal{H}_{X_{M},D}(M)\oplus L^{2}\mathcal{H}_{X_{M},\mathrm{co}}(M) \\
  L^{2}\Omega_{G}(M) &=&\mathcal{E}_{X_M}(M)\oplus \mathcal{C}_{X_M}(M)\oplus
\mathcal{H}_{X_{M},N}(M)\oplus L^{2}\mathcal{H}_{X_{M},\mathrm{ex}}(M)
\end{eqnarray*}
\end{corollary}

\begin{remark}
All the results above can be refined in terms of $\pm$-spaces,
for instance, 
$$\mathcal{H}^\pm_{X_{M},D}(M)\cong
\mathcal{H}^{n-\pm}_{X_{M},N}(M), \quad L^{2}\Omega^\pm_{G}(M)
=\mathcal{E}^\pm_{X_M}(M)\oplus \mathcal{C}^\pm_{X_M}(M)\oplus
\mathcal{H}^\pm_{X_{M},D}(M)\oplus
L^{2}\mathcal{H}^\pm_{X_{M},\mathrm{co}}(M)$$ 
\dots etc.
\end{remark}

\subsection{Relative and absolute $X_M$-cohomology }

Using $\d_{X_M}$ and $\delta_{X_M}$ we can form a number of
$\ZZ_2$-graded complexes. A $\ZZ_2$-graded complex is a pair of
Abelian groups $C^\pm$ with  homomorphisms between them:
\begin{center}
\begin{pspicture}(-2,-0.5)(2,0.5)
\rput(-1.3,0.05){$C^+$}
\rput(1.3,0.05){$C^-$}
\rput(0,0.4){$\d_+$} \rput(0,-0.4){$\d_-$}
\psline[arrowscale=1.5]{->}(-0.8,0.1)(0.8,0.1) \psline[arrowscale=1.5]{->}(0.8,-0.1)(-0.8,-0.1)\end{pspicture}
\end{center}
satisfying $\d_+\circ\d_- = 0 = \d_-\circ\d_+$.  The two (co)homology groups of such a complex are defined in the obvious way: $H^\pm = \ker\d_\pm/\im\d_\mp$.
The complexes we have in mind are,
\begin{eqnarray*}
(\Omega_G^\pm,\d_{X_M})&\quad&(\Omega_G^\pm,\delta_{X_M})\\
(\Omega_{G,D}^\pm,\d_{X_M}) && (\Omega_{G,N}^\pm,\delta_{X_M}).
\end{eqnarray*}
The two on the lower line are subcomplexes of the corresponding
upper ones, because $i^*$ commutes with $\d_{X_M}$. By analogy with the de Rham groups, we denote
$$H^\pm_{X_M}(M) := H^\pm(\Omega_G,\,\d_{X_M}) \quad \mbox{and} \quad
H^\pm_{X_M}(M,\,\partial M) := H^\pm(\Omega_{G,D},\,\d_{X_M}).$$

The decomposition theorems above lead to the following result.
\begin{theorem}[$X_{M}$-Hodge Isomorphism]\label{thm:X_M-Hodge}
Let $X\in\gg$.
There are the following isomorphisms of vector spaces:
\begin{enumerate}
\item[(a)] $H^{\pm}_{X_M}(M,\,\partial M) \cong \mathcal{H}^\pm_{X_{M},D}(M) \cong H^\pm(\Omega_G^\pm,\delta_{X_{M}})$;

\item[(b)]  $H^{\pm}_{X_M}(M)\cong \mathcal{H}^\pm_{X_{M},N}(M) \cong H^\pm(\Omega_{G,N}^\pm,\delta_{X_{M}})$;

\item[(c)]($X_{M}$-Poincar\'e-Lefschetz duality): The Hodge star operator $\star$  on $\Omega_{G}(M)$ induces an isomorphism
    $$H^\pm_{X_{M}}(M)\cong H^{n-\pm}_{X_{M}}(M,\,\partial M).$$
\end{enumerate}
\end{theorem}

\begin{proof} The proofs use the decomposition theorems above.
For the first isomorphism in (a), Theorem~\ref{thm:Hodge-Morrey} (the
$X_M$-Hodge-Morrey decomposition theorem) implies a unique splitting
of any $\gamma \in \Omega^\pm_{G,D}$ into,
$$\gamma=\d_{X_{M}} \alpha_{\gamma}+\delta_{X_{M}} \beta_{\gamma} + \kappa_{\gamma}$$
where $\d_{X_{M}} \alpha_{\gamma} \in \mathcal{E}^\pm_{X_M}(M)$,
$\delta_{X_{M}} \beta_{\gamma} \in \mathcal{C}^\pm_{X_M}(M)$ and
$\kappa_{\gamma} \in L^{2}\mathcal{H}^\pm_{X_{M}}(M)$. If
$\d_{X_{M}}\gamma=0$ then $\delta_{X_{M}} \beta_{\gamma}=0$, but
$i^{*}\gamma =0$ implies $i^{*}(\kappa_{\gamma})=0$ so that
$\kappa_{\gamma} \in \mathcal{H}^\pm_{X_{M},D}(M)$. Thus,
$$\gamma \in \ker \d_{X_{M}}\restrict{\Omega_{G,D}} \Longleftrightarrow
\gamma=\d_{X_{M}} \alpha_{\gamma}+\kappa_{\gamma}.$$
This establishes the isomorphism $H^\pm_{X_{M}}(M,\,\partial M) \cong \mathcal{H}^\pm_{X_{M},D}(M)$.

For the second isomorphism in (a), the second $X_M$-Hodge-Morrey-Friedrichs
decomposition of Corollary \ref{co.2.6} implies as
well  a unique splitting of any $\gamma \in \Omega^\pm_{G}(M)$ into,
$$\gamma=\d_{X_{M}}\xi_{\gamma} + \delta_{X_{M}}\eta_{\gamma} +\delta_{X_{M}}\zeta_{\gamma} + \lambda_{\gamma}$$
where $\d_{X_{M}}\xi_{\gamma} \in \mathcal{E}^\pm_{X_M}(M)\,$,
$\;\delta_{X_{M}} \eta_{\gamma} \in \mathcal{C}^\pm_{X_M}(M)\,$,
$\;\delta_{X_{M}} \zeta_{\gamma} \in
L^{2}\mathcal{H}^\pm_{X_{M},\mathrm{co}}(M)$ and $\lambda_{\gamma}
\in \mathcal{H}^\pm_{X_{M},D}(M)$.

If $\delta_{X_{M}}\gamma=0$, then $\d_{X_{M}}\xi_{\gamma}=0$, and hence
$$\gamma \in \ker \delta_{X_{M}}\Longleftrightarrow  \gamma=\delta_{X_{M}} (\eta_\gamma+\zeta_\gamma) + \lambda_{\gamma}.$$
This establishes the isomorphism $\mathcal{H}^\pm_{X_{M},D}(M) \cong H^\pm_{X_{M}}(\Omega_G^\pm,\delta_{X_{M}})$.

Part (b) is proved similarly, and part (c) follows from (a) and (b) and the fact that the Hodge star operator defines an isomorphism $\mathcal{H}^\pm_{X_M,D}(M)\cong \mathcal{H}^{n-\pm}_{X_M,N}(M)$.
\end{proof}

The theorem of Hodge is often quoted as saying that every (de Rham) cohomology class on a compact Riemannian manifold without boundary contains a unique harmonic form.  The corresponding statement for $X_M$-cohomology on a manifold with boundary is,

\begin{corollary}\label{coro.Hodge is often quoted}
Each absolute $X_M$-cohomology class contains a unique Neumann $X_M$-harmonic
field, and each relative $X_M$-cohomology class contains a unique
Dirichlet $X_M$-harmonic field.
\end{corollary}

\section{Relation with equivariant cohomology}
\label{sec:equivariant cohomology}

When the manifold in question has no boundary, Atiyah and Bott
\cite{AB} discuss the relationship between equivariant cohomology
and $X_M$-cohomology by using their localization theorem.  In this
section we will relate our relative and absolute $X_M$-cohomology
with the relative and absolute equivariant cohomology
$H_{G}^\pm(M,\partial M)$ and $H_{G}^\pm(M)$; the arguments are no
different to the ones in \cite{AB}. First we recall briefly the
basic definitions of equivariant cohomology, and the relevant
localization theorem, and then state the conclusions for the
relative and absolute $X_M$-cohomology.

If a torus $G$ acts on a manifold $M$ (with or without boundary), the Cartan model for the equivariant cohomology is defined as follows. Let $\{X_1,\dots,X_\ell\}$ be a basis of $\gg$ and $\{u_1,\dots,u_\ell\}$ the corresponding coordinates. The \emph{Cartan complex} consists of polynomial\footnote{we use real valued polynomials, though complex valued ones works just as well, and all tensor products are thus over $\RR$, unless stated otherwise} maps from $\gg$ to the space of invariant differential forms, so is equal to $\Omega^*_G(M)\otimes R$ where $R=\RR[u_1,\dots,u_\ell]$, with differential
$$\d_{\mathrm{eq}}(\omega) = \d\omega + \sum_{j=1}^\ell u_j\,\iota_{X_j}\omega.$$
The equivariant cohomology $H_G^*(M)$ is the cohomology of this complex. The relative equivariant cohomology $H_G^*(M,\partial M)$ (if $M$ has non-empty boundary) is formed by taking the subcomplex with forms that vanish on the boundary $i^*\omega=0$, with the same differential.

The cohomology groups are graded by giving the $u_i$ weight 2 and a $k$-form weight $k$, so the differential $\d_{\mathrm{eq}}$ is of degree 1. Furthermore,  as the cochain groups are $R$-modules, and $\d_{\mathrm{eq}}$ is a homomorphism of $R$-modules, it follows that the equivariant cohomology is an $R$-module.
The localization theorem of Atiyah and Bott \cite{AB} gives information on the module structure (there it is only stated for absolute cohomology, but it is equally true in the relative setting, with the same proof; see also Appendix C of \cite{Guillemin}).

First we define the following subset of $\gg$,
$$Z := \bigcup_{\widehat{K}\subsetneq G}\mathfrak{k}$$
where the union is over proper isotropy subgroups $\widehat{K}$ (and
$\mathfrak{k}$ its Lie algebra) of the action on $M$. If $M$ is
compact, then $Z$ is a finite union of proper subspaces of $\gg$.
Let $F = \mathrm{Fix}(G,M)=\{x \in M \mid G\cdot x=x\}$ be the set
of fixed points in $M$. It follows from the local structure of group
actions that $F$ is a submanifold of $M$, with boundary $\partial F
= F\cap \partial M$.

\begin{theorem}[Atiyah-Bott \protect{\cite[Theorem 3.5]{AB}}]\label{localization them.} The inclusion $j:F\hookrightarrow M$ induces homomorphisms of $R$-modules
$$H_G^*(M) \stackrel{j^*}{\longrightarrow} H_G^*(F)$$
$$H_G^*(M,\partial M) \stackrel{j^*}{\longrightarrow} H_G^*(F,\partial F)$$
whose kernel and cokernel have support in $Z$.
\end{theorem}

In particular, this means that if $f\in I(Z)$ (the ideal in $R$ of polynomials vanishing on $Z$) then the localizations\footnote{The localized ring $R_f$ consists of elements of $R$ divided by a power of $f$ and if $K$ is an $R$-module, its localization is $K_f:=K\otimes_R R_f$; they correspond to restricting to the open set where $f$ is non-zero. See the notes by Libine \cite{Libine} for a good discussion of localization in this context.} $H_G^*(M)_f$ and  $H_G^*(F)_f$ are isomorphic $R_f$-modules. Notice that the act of localization destroys the integer grading of the cohomology, but since the $u_i$ have weight 2, it preserves the parity of the grading, so that the separate even and odd parts are maintained: $H_G^\pm(M)_f \cong H_G^\pm(F)_f$. The same reasoning applies to the cohomology relative to the boundary, so $H_G^\pm(M,\partial M)_f \cong H_G^\pm(F,\partial F)_f$

Since the action on $F$ is trivial, it is immediate from the definition that there is an isomorphism of $R$-modules, $H_G^*(F)\cong H^*(F)\otimes R$ so that the localization theorem shows $j^*$ induces an isomorphism of $R_f$-modules,
\begin{equation} \label{eq:localized j^* isomorphism}
H_G^\pm(M)_f \stackrel{j^*}{\longrightarrow} H^\pm(F) \otimes R_f.
\end{equation}
It follows that $H_G^\pm(M)_f$ is a free $R_f$ module whenever $f\in
I(Z)$. Of course, analogous statements hold for the relative
versions. Since localization does not alter the rank of a module (it
just annihilates torsion elements), we have that
$$\rank H_G^\pm(M) = \dim H^\pm(F),\qquad \rank H_G^\pm(M,\partial M) = \dim H^\pm(F,\partial F).$$

For $X\in\gg$, define $N(X_M) = \{x\in M\mid X_M(x)=0\}$, the set of
zeros of the vector field $X_M$.  Since $X$ generates a torus action, $N(X_M)$ is a manifold with boundary $\partial N(X_M) = N(X_M)\cap \partial M$.  Clearly $N(X_M)\supset F$, and $N(X_M)= F$ if and only if $X\not\in Z$.

\begin{theorem}\label{relative equivariant}
Let  $X = \sum_j s_jX_j\in \gg$.  If the set of zeros of the
corresponding vector field $X_M$ is equal to the fixed point set for
the $G$-action (i.e. $N(X_M)=F$) then
\begin{equation}\label{eq.relative iso1.}
H_{X_M}^\pm(M,\,\partial M)\cong H_{G}^\pm(M,\partial
M)/\mathfrak{m}_X H_{G}^\pm(M,\partial M),
\end{equation}
and
\begin{equation}\label{eq.relative iso2.}
    H_{X_M}^\pm(M)\cong H_{G}^\pm(M)/\mathfrak{m}_X H_{G}^\pm(M)
\end{equation}
where $\mathfrak{m}_X = \left<u_1-s_1,\dots,u_l-s_l\right>$ is the ideal of polynomials vanishing at $X$.
\end{theorem}

\begin{proof}
Our assumption $N(X_M)=F$ is equivalent to  $X \in\gg\setminus Z$.
Therefore there is a polynomial $f \in I(Z)$ such that $f(X)\neq 0$. In addition, we can use $f$ and replace the ring $R$ by $R_f$ and then localize $ H_G^\pm(M)$ and $H_G^\pm(M,\partial M)$ to make $H_G^\pm(M)_f$ and $H_{G}^\pm(M,\partial M)_f$ which are free $R_f$-modules.

We now apply the lemma stated below, in which the left-hand side is obtained by putting $u_i=s_i$ before taking cohomology, so results in $H^\pm_{X_M}(M)$ (or similar for the relative case), while the right-hand side is the right-hand side of (\ref{eq.relative iso1.}) and  (\ref{eq.relative iso2.}), so proving the theorem.
\end{proof}

\begin{lemma}[Atiyah-Bott \protect{\cite[Lemma
5.6]{AB}}]\label{lemma 5.6} Let $(C^*,d)$ be a cochain complex of free
$R$-modules and assume that, for some polynomial $f$, $H(C^*,d)_f$
is a free module over the localized ring $R_f$. Then, if $s \in
\RR^l$ with $f(s)\neq 0$,$$ H^\pm(C_{s}^*,d_s)\cong H^\pm(C^*,d)\bmod
\mathfrak{m}_s
$$ where $\mathfrak{m}_s $ is the (maximal) ideal $\left<u_1-s_1,\dots,u_l-s_l\right>$ at $X$ in $\RR[\gg]$.
\end{lemma}

\begin{corollary} \label{coroll:Witten-fixed isomorphism}
Let $X \in\gg$ and $j_{X}:N(X_M)\hookrightarrow M$, then $j^*_{X}$
induces the following isomorphisms
\begin{itemize}
\item[1-] $H_{X_M}^\pm(M) \cong H^\pm(N(X_M))$,
\item[2-] $H_{X_M}^\pm(M,\partial M) \cong  H^\pm(N(X_M),\partial
N(X_M))$.
\end{itemize}
\end{corollary}

\begin{proof}
First suppose $X\not\in Z$. Then
the isomorphisms above follow by reducing equation
(\ref{eq:localized j^* isomorphism}) modulo $\mathfrak{m}_X$ and
applying Theorem \ref{relative equivariant}.

If on the other hand, $X \in Z$, then let $G'$ be the corresponding isotropy
subgroup, so that $N(X_M)=F':=\textrm{Fix}(G',M)$ (it is clear that
$G'\supset G(X)$, the subgroup of $G$ generated by $X$). The
considerations above show that $H^\pm_{X_M,G'}(M,\partial M)\cong
H^\pm(F',\partial F')$ and $H^\pm_{X_M,G'}(M)\cong H^\pm(F')$,
where $H^\pm_{X_M,G'}(M)$ and $H^\pm_{X_M,G'}(M,\partial M)$ are defined using $G'$-invariant forms, and $\mathfrak{m}_{G',X}$ is the maximal
ideal at $X$ in the ring $\RR[\gg']$. Moreover, all classes in
$H^\pm_{X_M,G'}(M)$ and $H^\pm_{X_M,G'}(M,\partial M)$ have
representatives which are $G$-invariant, not only $G'$-invariant
(either by an averaging argument, or by using the unique
$X_M$-harmonic representatives). So, this gives
$H^\pm_{X_M,G}(M)\cong H^\pm_{X_M,G'}(M)$ and
$H^\pm_{X_M,G}(M,\partial M)\cong H^\pm_{X_M,G'}(M,\partial M)$,
$\forall X\in\gg'\subset\gg$ as desired.
\end{proof}

\begin{remark}\label{rem.$X_M$-cohomology and singular homology}
If $M$ is a compact manifold with boundary then $H^k(M)\cong H_{k}(M)$ and $H^k(M,\partial M) \cong H_{k}(M,\partial M)$, where $H_{k}(M)$ and $H_{k}(M,\partial M)$ are the absolute and relative singular homology with real coefficients. We observe that this fact together with corollary \ref{coroll:Witten-fixed isomorphism} give us the following isomorphisms
$$H_{X_M}^\pm(M) \cong H_{\pm}(N(X_M))\quad \mbox{and} \quad
  H_{X_M}^\pm(M,\partial M) \cong  H_{\pm}(N(X_M),\partial
N(X_M)),$$
where $H_{+}(N(X_M))=\oplus_i H_{2i}(N(X_M))$ and
$H_{-}(N(X_M),\partial N(X_M))=\oplus_i H_{2i+1}(N(X_M),\partial N(X_M))$,
by using the map
\begin{equation}\label{eq.inte.}
   [\omega]_{X_M}(\{c\})=\int_{c}j^*\omega,
\end{equation}
where $\omega$ is $X_M$-closed $\pm$-form representing the absolute
(or relative) $X_M$-cohomology class $[\omega]_{X_M}$ on $M$ and $c$
is a $\pm$-cycle representing the absolute (or relative) singular
homology class $\{c\}$ on $N(X_M)$. In this light, eq. (\ref{eq.Stokes'
theorem}), corollary \ref{coroll:Witten-fixed isomorphism} and the
bijection (\ref{eq.inte.}) prove the following statement:

\emph{An $X_M$-closed form $\omega$ is $X_M$-exact iff all the
periods of $j^*\omega$ over all $\pm$-cycles of $N(X_M)$ vanish.}
\end{remark}

\section{Interior and boundary subspaces}
\label{sec.style of DeTurck-Gluck}

In this section we visit some recent work of DeTurck and Gluck
\cite{Gluck} on harmonic fields and cohomology (see also
\cite{Clay1,Clay2} for details), and adapt it to $X_M$-harmonic
fields.

\subsection{Interior and boundary subspaces after DeTurck and Gluck}
Given the usual manifold $M$ with boundary, there is a long exact sequence in cohomology associated to the pair $(M,\partial M)$ and one can use this to define two subspaces of $H^k(M)$ and $H^k(M,\partial M)$ as follows:
\begin{itemize}
\item the \emph{interior} subspace $IH^k(M)$ of $H^k(M)$ is the kernel of $i^*:H^k(M)\to H^k(\partial M)$
\item the \emph{boundary} subspace  $BH^k(M,\partial M)$ of $H^k(M,\partial M)$  is the image of $\d:H^{k-1}(\partial M)\to H^k(M,\partial M)$
\end{itemize}
Note that if $M$ has no boundary, then $IH^k=H^k$ and $BH^k=0$, as should be expected from their names.

At the level of cohomology there is no `natural' definition for the boundary part of the absolute cohomology nor the interior part of the relative cohomology.  However, DeTurck and Gluck \cite{Gluck} use the metric and harmonic representatives to provide these.  Firstly the subspaces defined above are realized as
\begin{eqnarray*}
\mathcal{IH}_N^k &=&
\{\omega\in \mathcal{H}^k_{N}(M)\mid i^*\omega=\d\theta, \mbox{ for some }
\theta\in\Omega^{k-1}(\partial M)\}\\
\mathcal{BH}^k_D &=&   \mathcal{H}^k_{D}(M)\cap\mathcal{H}^k_{\mathrm{ex}}
\end{eqnarray*}
respectively (these are denoted $\mathcal{E}_{\partial}\mathcal{H}^k_{N}(M)$ and $\mathcal{EH}^k_{D}(M)$ respectively in \cite{Gluck,Clay1,Clay2}). They then use the Hodge star operator to define the other spaces:
\begin{eqnarray*}
\mathcal{BH}_N^k &=& \mathcal{H}^{k}_{N}(M)\cap\mathcal{H}^k_{\mathrm{co}}\\
\mathcal{IH}^k_D &=& \{\omega\in \mathcal{H}^k_{D}(M):
i^*\star\omega=\d\kappa, \mbox{ for some }
\kappa\in\Omega^{n-k-1}(\partial M)\}
\end{eqnarray*}
(denoted  $c\mathcal{EH}^k_{N}(M)$ and $c\mathcal{E}_{\partial}\mathcal{H}^k_{D}(M)$ in \cite{Gluck, Clay1,Clay2}). The first theorem of DeTurck and Gluck on this subject is

\begin{theorem}[DeTurck and Gluck \cite{Gluck}]\label{thm.DE-Gluck decompo.}
Both $\mathcal{H}_D^k$ and $\mathcal{H}_N^k$ have orthogonal decompositions,
\begin{eqnarray*}
  \mathcal{H}^k_{N}(M)&=& \mathcal{IH}_N^k \oplus \mathcal{BH}_N^k\\
  \mathcal{H}^k_{D}(M)&=&\mathcal{BH}_D^k\oplus \mathcal{IH}_D^k.
\end{eqnarray*}
Furthermore, the two boundary subspaces are mutually orthogonal inside $L^2\Omega$.
\end{theorem}

However the interior subspaces are not orthogonal, and they prove

\begin{theorem}[DeTurck-Gluck \cite{Gluck}]\label{De-Glu duality angles}
The principal angles between the interior subspaces $\mathcal{IH}_N^k$ and $\mathcal{IH}_D^k$ are all acute.
\end{theorem}

Part of the motivation for considering these principal angles, called \emph{Poincar\'e duality angles}, is that they should measure in some sense how far the Riemannian manifold $M$ is from being closed.  That these angles are non-zero follows from the fact that $\mathcal{H}_N^k\cap \mathcal{H}_D^k=0$, see \cite{Schwarz}. Another consequence of this, pointed out by DeTurck and Gluck is that the Hodge-Morrey-Freidrichs decomposition  can be refined to a 5-term decomposition,
\begin{equation}\label{eq. Det&Gluck decom.}
   \Omega^k(M)=\d\Omega_D^{k-1} \oplus \delta\Omega_N^{k+1}
\oplus(\mathcal{H}^k_D + \mathcal{H}^{k}_N)\oplus
\mathcal{H}^k_{\mathrm{ex,co}},
\end{equation}
where
$\mathcal{H}^k_{\mathrm{ex,co}}=\mathcal{H}^k_{\mathrm{ex}}\cap\mathcal{H}^k_{\mathrm{co}}$
and the symbol $+$ indicates a direct sum whereas $\oplus$ indicates an orthogonal direct sum.

In his thesis \cite{Clay1}, Shonkwiler  measures these Poincar\'e duality angles in interesting examples of manifolds with boundary derived from complex projective spaces and Grassmannians and shows that in these examples the angles do indeed tend to zero as the boundary shrinks to zero, see alternatively \cite{Clay2}.

\subsection{Extension to $X_M$-cohomology}

It seems reasonable to think that we can extend further to the style
of DeTurck-Gluck, and break down the Neumann and Dirichlet
$X_M$-harmonic fields into interior and boundary subspaces. If so,
does the natural extension of corollary \ref{coroll:Witten-fixed
isomorphism} hold? The answer is affirmative and contained in the
proof of theorem \ref{thm:refine localization}.

Answering this question will indeed give more concrete understanding
of these isomorphisms and consequently will give a precise extension
to Witten's results when $\partial M\neq\emptyset$ (see Section
\ref{sec:conclusions}).

\paragraph{Refinement of the $X_M$-Hodge-Morrey-Friedrichs decomposition}
In \cite{Q&J}, we prove that
$$\mathcal{H}^{\pm}_{X_{M},N}(M) \cap \mathcal{H}^{\pm}_{X_{M},D}(M)=\{0\},$$
which implies that the sum $\mathcal{H}^{\pm}_{X_{M},N}(M) + \mathcal{H}^{\pm}_{X_{M},D}(M)$ is
a direct sum, and by using Green's formula~(\ref{eq.2.16}), one finds that the orthogonal complement of $\mathcal{H}^{\pm}_{X_{M},N}(M) + \mathcal{H}^{\pm}_{X_{M},D}(M)$ inside $\mathcal{H}^{\pm}_{X_{M}}(M)$  is
$\mathcal{H}^{\pm}_{X_{M},\mathrm{ex,co}}(M)=\mathcal{H}^{\pm}_{X_{M},\mathrm{ex}}(M) \cap \mathcal{H}^{\pm}_{X_{M},\mathrm{co}}(M)$.
Therefore, we can refine the $X_M$-Friedrichs decomposition (theorem \ref{thm:X_M-Friedrichs}) into
$$\mathcal{H}^{\pm}_{X_{M}}(M) = ( \mathcal{H}^{\pm}_{X_{M},N}(M)+
\mathcal{H}^{\pm}_{X_{M},D}(M) )\oplus
\mathcal{H}^{\pm}_{X_{M},\mathrm{ex,co}}(M). $$
Consequently, following DeTurck and Gluck's decomposition (\ref{eq. Det&Gluck decom.}), we can refine the $X_M$-Hodge-Morrey-Friedrichs decompositions (Corollary \ref{co.2.6}) into the following five terms decomposition:
\begin{equation}\label{eq:5 term X_M decomposition}
\Omega^{\pm}_{G}(M) =\mathcal{E}^{\pm}_{X_M}(M)\oplus
\mathcal{C}^{\pm}_{X_M}(M)\oplus ( \mathcal{H}^{\pm}_{X_{M},N}(M)+
\mathcal{H}^{\pm}_{X_{M},D}(M) )\oplus
\mathcal{H}^{\pm}_{X_{M},\mathrm{ex,co}}(M).
\end{equation}
Here as usual, $\oplus$ is an orthogonal direct sum, while $+$ is just a direct sum.

\paragraph{Interior and boundary portions of $X_M$-cohomology}
Following the ordinary case described above, we can define interior
and boundary portions of the $X_M$-cohomology and $X_M$-harmonic
fields by
\begin{equation} \label{eq:interior-boundary for X_M cohomology}
\begin{array}{rcl}
IH_{X_M}^\pm(M) &=& \ker[i^*:H^\pm_{X_M}(M)\to H^\pm_{X_M}(\partial M)]\\[4pt]
BH_{X_M}^\pm(M,\partial M) &=& \im[\d_{X_M}:H_{X_M}^\mp(\partial M) \to H_{X_M}^\pm(M,\partial M)].
\end{array}
\end{equation}
Here $\d_{X_M}$ is the standard construction: given a closed form $\lambda$ on $\partial M$, let $\tilde\lambda$ be an extension to $M$. Then $\d_{X_M}\tilde\lambda$ defines a well-defined element of $H_{X_M}(M,\partial M)$.  These spaces are realized through corollary \ref{coro.Hodge is often quoted} as
\begin{eqnarray*}
\mathcal{IH}_{X_M,N}^\pm &=&
\{\omega\in \mathcal{H}^\pm_{X_M,N}(M)\mid i^*\omega=\d_{X_M}\theta, \mbox{ for some }
\theta\in\Omega^\mp(\partial M)\}\\
\mathcal{BH}^\pm_{X_M,D} &=&   \mathcal{H}^\pm_{X_M,D}(M)\cap\mathcal{H}^\pm_{X_M,\mathrm{ex}}
\end{eqnarray*}
respectively. Now use the Hodge star operator to define the other spaces:
\begin{eqnarray*}
\mathcal{IH}^\pm_{X_M,D} &=& \{\omega\in \mathcal{H}^\pm_{X_M,D}(M):
i^*\star\omega=\d_{X_M}\kappa, \mbox{ for some }
\kappa\in\Omega^{n-\mp}(\partial M)\}\\
\mathcal{BH}_{X_M,N}^\pm &=&
\mathcal{H}^{\pm}_{X_M,N}(M)\cap\mathcal{H}^\pm_{X_M,\mathrm{co}}.
\end{eqnarray*}
Note that Hodge star maps boundary to boundary and interior to interior; it follows that, for example $\mathcal{BH}_{X_M,N}^\pm \cong \mathcal{BH}_{X_M,D}^{n-\pm}$.

\begin{theorem}\label{thm. largest orthogonal }
The boundary subspace $\mathcal{BH}^{\pm}_{X_{M},N}(M)$
is the largest subspace of\/ $\mathcal{H}^{\pm}_{X_{M},N}(M)$ orthogonal to all of
$\mathcal{H}^{\pm}_{X_{M},D}(M)$ while the boundary subspace
$\mathcal{BH}^\pm_{X_{M},D}(M)$ is the largest
subspace of\/ $\mathcal{H}^{\pm}_{X_{M},D}(M)$  orthogonal to all of
$\mathcal{H}^{\pm}_{X_{M},N}(M).$
\end{theorem}

\begin{proof}
The orthogonality follows immediately from Green's formula~(\ref{eq.2.16}) while the rest of the proof follow immediately from the $X_M$-Friedrichs decomposition theorem (theorem \ref{thm:X_M-Friedrichs}) (restricted to smooth invariant forms).
\end{proof}

The main goal of this subsection is to prove the following theorem
and to answer the question above.
\begin{theorem}\label{thm.De- Gluck and ours decom.}
Analogous to theorem~\ref{thm.DE-Gluck decompo.}, we have the orthogonal decompositions
\begin{eqnarray*}
  \mathcal{H}^\pm_{X_M,N}(M)&=& \mathcal{IH}_{X_M,N}^\pm \oplus \mathcal{BH}_{X_M,N}^\pm\\
  \mathcal{H}^k_{X_M,D}(M)&=&\mathcal{BH}_{X_MD}^\pm \oplus \mathcal{IH}_{X_M,D}^\pm.
\end{eqnarray*}
\end{theorem}

\begin{remark}
The proof by DeTurck and Gluck of the analogous result uses the duality between de Rham cohomology and singular homology.  However, we do not have such a result on $M$ (though perhaps a proof using the equivariant homology described in \cite{MacPherson} would be possible), so we give a direct proof involving only the cohomology---the same argument can be used to prove DeTurck and Gluck's original theorem (replacing $\pm$ by $k$ everywhere). An alternative argument can be given using the localization to the fixed point set (corollary~\,\ref{coroll:Witten-fixed isomorphism})---details of which can be found in \cite{my thesis}.
\end{remark}

\begin{proof}
The orthogonality of the right hand sides follows from Green's formula (\ref{eq.2.16}).
It follows that
\begin{equation}\label{eq:direct sum subset}
\mathcal{IH}_{X_M,N}^\pm \oplus \mathcal{BH}_{X_M,N}^\pm \subset \mathcal{H}^\pm_{X_M,N}(M)\quad\mbox{and}\quad
  \mathcal{BH}_{X_MD}^\pm \oplus \mathcal{IH}_{X_M,D}^\pm \subset \mathcal{H}^k_{X_M,D}(M).
\end{equation}

Now consider the long exact sequence in $X_M$-cohomology derived from the inclusion $i:\partial M\hookrightarrow M$,
$$\cdots\stackrel{i^*}{\longrightarrow} H_{X_M}^\mp(\partial M) \stackrel{\d_{X_M}}{\longrightarrow} H_{X_M}^\pm(M,\partial M) \stackrel{\rho^*}{\longrightarrow} H_{X_M}^\pm(M) \stackrel{i^*}{\longrightarrow} H_{X_M}^\pm(\partial M) \stackrel{\d_{X_M}}{\longrightarrow} H_{X_M}^\mp(M,\partial M)\longrightarrow\cdots
$$
It follows from the exactness that
$$IH_{X_M}^\pm(M) = \im\rho^*,\quad \mbox{and} \quad BH_{X_M}^\pm(M,\partial M)=\ker\rho^*.$$
Thus,
$H^\pm_{X_M}(M,\partial M) \cong BH_{X_M}^\pm(M,\partial M) + IH_{X_M}^\pm(M)$, (direct sum) or equivalently
\begin{equation}\label{eq:direct sum cohomology}
\mathcal{H}^\pm_{X_M,D} \cong \mathcal{BH}^\pm_{X_M,D} + \mathcal{IH}^\pm_{X_M,N}.
\end{equation}
It follows from equations~(\ref{eq:direct sum subset}) and (\ref{eq:direct sum cohomology}) that $\dim(\mathcal{IH}^\pm_{X_M,D})\leq\dim(\mathcal{IH}^\pm_{X_M,N})$.
However, the Hodge star operator identifies $\mathcal{IH}_{X_M,N}^\pm$ with $\mathcal{IH}_{X_M,D}^{n-\pm}$ which implies that the inequality in dimensions is in fact an equality, and the result follows.
\end{proof}

\begin{theorem}\label{thm:refine localization}
 Let $F'=N(X_M)$. We have isomorphisms,
$$\begin{array}{cc}
\mathcal{IH}^\pm_{X_{M},N}(M)\cong \mathcal{IH}^\pm_{N}(F'), 
&\mathcal{BH}^{\pm}_{X_{M},D}(M)\cong \mathcal{BH}^{\pm}_{D}(F'),\\ \mathcal{IH}^\pm_{X_{M},D}(M)\cong\mathcal{IH}^\pm_{D}(F'), 
& \mathcal{BH}^\pm_{X_{M},N}(M)\cong \mathcal{BH}^\pm_{N}(F').
\end{array}
$$
\end{theorem}

\begin{proof}
We prove the first two; the other two follow by applying the Hodge star operator (on $M$ and on $F'$).  Denote by $j_X$ the inclusion of the pair, $j_X:(F',\partial F') \hookrightarrow (M,\partial M)$. Then $j_X$ induces a chain map between the long exact sequences of $X_M$ cohomology on $M$ and de Rham cohomology on $F'$, which by corollary \ref{coroll:Witten-fixed isomorphism} is an isomorphism.

Since the interior part of the absolute cohomology and the boundary part of the relative cohomology are defined from these long exact sequences, it follows that $j_X$ induces isomorphisms
$$IH_{X_M}(M)^\pm \cong IH^\pm(F'),\quad \mbox{and} \quad BH_{X_M}^\pm(M,\partial M) \cong BH^\pm(F',\partial F').$$
It then follows from the $X_M$-Hodge theorem \ref{thm:X_M-Hodge} that there are isomorphisms
$\mathcal{IH}_{X_M,N}(M) \cong \mathcal{IH}_N^\pm(F')$ and 
$ \mathcal{BH}_{X_M,D}^\pm(M) \cong \mathcal{BH}_D^\pm(F').$
\end{proof}

The analogue of Gluck and DeTurck's theorem for the Poincar\'e
duality angles (theorem \,\ref{De-Glu duality angles}) also holds.
The $X_M$-Poincar\'e duality angles are defined in the obvious way,
as the principal angles between $\mathcal{IH}^\pm_{X_M,D}$ and
$\mathcal{IH}^\pm_{X_M,N}$.

\begin{proposition}\label{prop. our duality angles}
The $X_M$-Poincar\'{e} duality angles are all acute.
\end{proposition}
\begin{proof}
These angles can be neither 0 nor $\pi/2$, firstly because
$\mathcal{H}^{\pm}_{X_{M},N}(M)\cap\mathcal{H}^{\pm}_{X_{M},D}(M)=\{0\}$ (shown
in \cite{Q&J}), and secondly because of theorem \ref{thm. largest orthogonal }. Hence they
must all be acute.
\end{proof}

The results above and  in \cite{Q&J} would allow us to extend most of the results of \cite{Clay1} to the context of $X_M$-cohomology and $X_M$-Poincar\'{e} duality angles but we leave this for future work.

\vskip 1cm

\section{Conclusions}
\label{sec:conclusions}

In previous sections, we began with the action of a torus $G$; here
we state results for a given Killing vector field $K$ on a compact
Riemannian manifold $M$ (with or without boundary), more in keeping
with Witten's original work \cite{Witten}.  Recall that the group
$\mathrm{Isom}(M)$ of isometries of $M$ is a compact Lie group, and
the smallest closed subgroup $G(K)$ containing $K$ in its Lie
algebra is Abelian, so a torus. Furthermore, the submanifold $N(K)$
of zeros of $K$ coincides with $\mathrm{Fix}(G(K),M).$

The equivariant cohomology constructions of Section
\ref{sec:equivariant cohomology} give us the proof of the following
result, which extends the theorem of Witten (our Theorem
\ref{thm:fixed point set}) to manifolds with boundary.
\begin{theorem}\label{Witten-De rham}
Let $K$ be a Killing vector field on the compact Riemannian manifold
$M$ (with or without boundary), and let $N(K)$ be the submanifold of
zeros of $K$. Then pullback to $N$ induces isomorphisms
$$H_{K}^\pm(M) \cong H^\pm(N(K)),
\quad\mbox{and}\quad H_{K}^\pm(M,\,\partial M) \cong
H^\pm(N(K),\,\partial N(K)).
$$
\end{theorem}

\begin{proof}
Apply Corollary \ref{coroll:Witten-fixed isomorphism} to the equivariant cohomology for the action of the torus $G(K)$.
\end{proof}

Furthermore, using the Hodge star operator, the Poincar\'e-Lefschetz duality of Theorem \ref{thm:X_M-Hodge}(c) corresponds under the isomorphisms in the theorem above, to Poincar\'e-Lefschetz duality on the fixed point space.

Translating this theorem into the language of harmonic fields, shows
\begin{equation}\label{eq. witten exten.on X_M-harm.}
    \mathcal{H}^\pm_{K,N}(M)\cong \mathcal{H}^\pm_{N}(N(K))
\quad\mbox{and}\quad \mathcal{H}^\pm_{K,D}(M)\cong
\mathcal{H}^\pm_{D}(N(K)).
\end{equation}
where $\mathcal{H}^\pm_{N}(N(K))$ and $\mathcal{H}^\pm_{D}(N(K)) $ are the ordinary Neumann and Dirichlet harmonic fields on $N(K)$ respectively. 
The fact that theorem \ref{Witten-De rham} and eq. (\ref{eq. witten
exten.on X_M-harm.}) can be refined to the style of
theorem~\ref{thm:refine localization} which gives a more precise
meaning for these isomorphisms.

\begin{corollary}\label{coro.Witt. Diri-Neum}
Given any harmonic field on $N(K)$ with either Dirichlet or Neumann
boundary conditions, there is a unique $K$-harmonic field on $M$
with the corresponding boundary conditions whose restriction on $N(K)$ is cohomologous to the given field.
\end{corollary}

Note that if $\partial N(K)=\emptyset$ then the boundary condition
on $N(K)$ is non-existent, and so every harmonic form (= field) on
$N(K)$ has corresponding to it both a unique Dirichlet and a unique
Neumann $K$-harmonic field on $M$.  Moreover, since in this case
there is no boundary part of the cohomology of $N(K)$, it follows
from theorem~\ref{thm:refine localization} that
$\mathcal{BH}_{X_M,N}=\mathcal{BH}_{X_M,D}=0$.

In other words, it means that all the de Rham cohomology of $N(K)$ must come only from the interior portion, i.e. $H^{\pm}(N(K))=H^{\pm}(N(K),\partial N(K))$,
which shows that every interior de Rham cohomology class has
corresponding to it both a unique relative and a unique absolute
$K$-cohomology class on $M$.

As an application, we have the fact that theorem \ref{Witten-De
rham} and corollary \ref{coro.Witt. Diri-Neum} can be used to extend
the other results of Witten in \cite{Witten} and we hope that this
extension will be useful in quantum field theory and other
mathematical and physical applications when $\partial
M\neq\emptyset$.

\paragraph{Euler characteristics} As is well known, given a complex of $\RR[s]$ (or $\mathbb{C}[s]$) modules whose cohomology is finitely generated, the Euler characteristic of the complex is independent of $s$. This remains true for a $\ZZ_2$-graded complex, for the same reasons (briefly, the cohomology is the direct sum of a torsion module and a free module, and the torsion cancels in the Euler characteristic).

Applying this to the complexes for $X_M$-cohomology, with $X_M=sK$,
it follows that $\chi(M)=\chi(N)$ and $\chi(M,\partial M) =
\chi(N,\partial N)$ (where $N=N(K)$), and furthermore applying the
same arguments to the manifold $\partial M$, one has  $\chi(\partial
M)=\chi(\partial N)$, i.e.
$$\chi(M)=\chi(\partial M)+\chi(M,\partial M)= \chi(\partial
N)+\chi(N,\partial N)=\chi(N).$$

\paragraph{Other Applications:}
We have shown that the Witten-Hodge theory can shed light to give
additional equivariant geometric and topological insight. In addition, the fact that we can use the new decompositions
of $L^2\Omega^{\pm}_G(M)$ given in
theorem~\ref{thm:Hodge-Morrey} and corollary \ref{co.2.6} and also
the relation between the $X_M$-cohomology and $X_M$-harmonic fields
(theorem \ref{thm:X_M-Hodge}) as powerful tools (under topological
aspects) in the theory of differential equations on
$L^2\Omega^{\pm}_G(M)$ to obtain the solubility of various
\textsc{bvp}s. In particular, we can extend most of the results of
chapter three of \cite{Schwarz} on $L^2\Omega^{\pm}_G(M)$ to the
context of the operators $\d_{X_M},$ $\delta_{X_M}$ and
$\Delta_{X_M}$. Moreover, the classical Hodge theory plays a
fundamental role in incompressible hydrodynamics and it has
applications to many other area of mathematical physics and
engineering \cite{Marsden}. So, following these, we hope that the
Witten-Hodge theory will be using as tools in these applications as
well.

\paragraph{Geometric question:} Finally, we proved that 
 $\mathcal{IH}^\pm_{X_{M},N}(M)\cong\mathcal{IH}^\pm_{N}(N(X_M))$ and  $\mathcal{IH}^\pm_{X_{M},D}(M)\cong \mathcal{IH}^\pm_{D}(N(X_M))$ 
and that the principal angles between the corresponding interior subspaces are all acute.
Hence, it would be interesting to answer the following

\medskip

\noindent\emph{How do the $X_M$-Poincar\'{e} duality angles between the interior subspaces
$\mathcal{IH}^\pm_{X_{M},N}(M)$ and $\mathcal{IH}^\pm_{X_{M},D}(M)$ depend on $X$, and how do they compare to the Poincar\'{e} duality angles between the interior subspaces
$\mathcal{IH}^\pm_{N}(N(X_M))$ and $\mathcal{IH}^\pm_{D}(N(X_M))$.}


\begin{thebibliography}{10}

\bibitem{Marsden}
R.~Abraham, J.E.~Marsden, and T.S.~Ratiu.
\newblock {\em Manifolds, tensor analysis, and applications}, volume~75 of {\em
  Applied Mathematical Sciences}.
\newblock Springer-Verlag, New York, Second Edition, 1988.

\bibitem{AB}
M.F.~Atiyah and R.~Bott.
\newblock The moment map and equivariant cohomology.
\newblock {\em Topology}, 23(1):1--28, 1984.

\bibitem{my thesis}
Q. S. A. Al-Zamil, {Algebraic topology of PDES.}  PhD Thesis,
Manchester Institute for Mathematical Sciences, School of
Mathematics, University of Manchester. In preparation.

\bibitem{Q&J}
Q. S. A. Al-Zamil and J. Montaldi, {Generalized Dirichlet to Neumann
operator on invariant differential forms and equivariant
cohomology}. http://eprints.ma.man.ac.uk/1528/. 2010.

\bibitem{Belishev}
M. Belishev and V. Sharafutdinov.
\newblock Dirichlet to {N}eumann operator on differential forms
\newblock \textit{Bull.\ Sci.\ Math.}, 132 :128--145, 2008

\bibitem{Gluck}
 D. DeTurck, H. Gluck, { Poincar\'{e} duality angles and Hodge
decomposition for Riemannian manifolds,} Preprint, 2004.

\bibitem{D&S}
G.F.D. Duff and D.C. Spencer, {Harmonic tensors on Riemannian
manifolds with boundary,} \emph{Ann.of Math}. \textbf{56}, 128--156,
1952.

\bibitem{Friedrichs}
K.O.~Friedrichs.
\newblock Differential forms on Riemannian manifolds.
\newblock {\em Comm. Pure Appl. Math.}, 8:551--590, 1955.

\bibitem{Guillemin}
V.~ Guillemin, V.~Ginzburg, and Y.~Karshon.
\newblock {\em Moment maps, cobordisms, and {H}amiltonian group actions},
  volume~98 of {\em Mathematical Surveys and Monographs}.
\newblock American Mathematical Society, Providence, RI, 2002.

\bibitem{Hodge}
W.V.D.~Hodge.
\newblock A {D}irichlet problem for harmonic functionals, with applications to
  analytic varieties.
\newblock {\em Proc.\ London Math.\ Soc.}, s2-36(1):257--303, 1934.

\bibitem{Lars}
L.~H{\"o}rmander.
\newblock {\em The analysis of linear partial differential operators. {III}},
  volume 274 of {\em Grundlehren der Mathematischen Wissenschaften [Fundamental
  Principles of Mathematical Sciences]}.
\newblock Springer-Verlag, Berlin, 1985.

\bibitem{Libine}
M.~Libine.
\newblock Lecture notes on equivariant cohomology.
\newblock {\em arXiv:0709.3615}, 2007.

\bibitem{Morrey}
C.B.~Morrey, Jr.
\newblock A variational method in the theory of harmonic integrals. {II}.
\newblock {\em Amer.\ J.\ Math.}, 78:137--170, 1956.

\bibitem{MacPherson}
R.~MacPherson, Equivariant invariants and linear geometry. \emph{Geometric Combinatorics}, 317--388, IAS/Park City Math. Ser., 13, Amer.\ Math.\ Soc., Providence, RI, 2007.

\bibitem{Schwarz}
G.~Schwarz,
\newblock {\em Hodge decomposition---a method for solving boundary value
  problems}, volume 1607 of {\em Lecture Notes in Mathematics}.
\newblock Springer-Verlag, Berlin, 1995.

\bibitem{Clay1}
C. Shonkwiler, {Poincar\'{e} Duality Angles for Riemannian Manifold
with boundary}. PhD Thesis, University of Pennsylvania.
http://www.math.upenn.edu/grad/dissertations/ShonkwilerDissertation.pdf
(2009).

\bibitem{Clay2}
C. Shonkwiler, Poincar\'e duality angles for Riemannian manifolds with boundary. Preprint. \textit{ArXiv:0909.1967}


\bibitem{Witten}
E.~Witten.
\newblock Supersymmetry and {M}orse theory.
\newblock {\em J. Differential Geom.}, 17(4):661--692, 1982.

\end{thebibliography}
\end{document}